\documentclass{imsart}

\arxiv{arXiv:2004.05028}

\startlocaldefs
\usepackage{amsmath}
\usepackage{amsthm}
\usepackage{amsfonts}
\usepackage{amssymb}
\usepackage[margin=1in]{geometry}
\theoremstyle{definition}
\usepackage{graphicx}

\newtheorem{theorem}{Theorem}[section]
\newtheorem{corollary}[theorem]{Corollary}

\newtheorem{lemma}[theorem]{Lemma}

\newtheorem{remark}[theorem]{Remark}

\usepackage{enumitem}
\numberwithin{equation}{section}
\usepackage{color}
\usepackage{wasysym}
\usepackage{morefloats}
\usepackage[square,numbers]{natbib}

\newcommand{\sign}{\text{sign}}

\endlocaldefs

\begin{document}

\begin{frontmatter}

\title{Minimal $\mathcal L^p$-Densities with Prescribed Marginals}

\runtitle{Minimal Densities with Prescribed Marginals}

\begin{aug}

\author{\fnms{Paolo} \snm{Guasoni}\thanksref{a}\ead[label=e1]{paolo.guasoni@dcu.ie}}
\and
\author{\fnms{Eberhard} \snm{Mayerhofer}\corref{}\thanksref{b}\ead[label=e2]{eberhard.mayerhofer@ul.ie}}
\and
\author{\fnms{Mingchuan} \snm{Zhao}\thanksref{b}\ead[label=e3]{mingchuan.zhao@ul.ie}}

\address[a]{Dublin City University, School of Mathematical Sciences, Glasnevin, Dublin, Ireland. \newline \printead{e1}}
\address[b]{University of Limerick, Department of Mathematics and Statistics, Castletroy, Ireland. \newline \printead{e2,e3}}
\runauthor{Guasoni, Mayerhofer and Zhao}

\end{aug}

\begin{abstract}

We derive sharp lower bounds for $\mathcal L^p$-functions on the $n$-dimensional unit hypercube in terms of their $p$-ths marginal moments. Such bounds are the unique solutions of a system of constrained nonlinear integral equations depending on the marginals. For square-integrable functions, the bounds have an explicit expression in terms of the second marginals moments.
\end{abstract}




\end{frontmatter}

\section{Introduction}

This paper obtains lower bounds of $\mathcal L^p$-functions on the $n$-dimensional unit hypercube in terms of moments of their marginals (i.e., one-dimensional projections). The lower bound is identified as the solution of a system of constrained nonlinear integral equations, where the marginals appear as inhomogeneous data. In the Hilbertian case $p=2$ (and only in such a case), the integral equations become linear and admit an explicit solution (Section \ref{sec: l2}), implying that any square-integrable function $g: \mathbb{R}^n \ni \xi \mapsto g(\xi) \in \mathbb{R}$ such that $\int g(\xi)d\xi = 1$, satisfies the bound
\begin{equation}\label{eq:l2bound}
\int g^2(\xi)d\xi\geq 
\left(\sum_{i=1}^n \int g_i(\xi_i)^2 d\xi_i\right) -(n-1)
,
\end{equation}
where $\xi = (\xi_i)_{1\le i\le n}$ and the marginal $g_i$ is the integral of $g$ with respect to all but the $i$-ths argument $\xi_i$, $1\leq i\leq n$. Henceforth, all integrals are understood on the unit (hyper)cube of the respective integration variable, unless explicitly stated otherwise.

This bound is reminiscent of the lower Fr\'echet-Hoeffding bound \citep[Theorem 2.2.3]{nelsen} for a copula $C=C(\xi)$ (i.e., the joint cumulative distribution function of $n$-dimensional random variable with uniform marginals),
\begin{equation}\label{eq frechet}
C(\xi)\geq \max\left\{\sum_{i=1}^n \xi_i-(n-1),0\right\}.
\end{equation}
However, a closer inspection shows that such a bound is rather different from the ones considered in this paper. First, \eqref{eq frechet} is a point-wise bound (which reduces to state mere positivity on the set $\{\sum_{i=1}^n \xi_i<n-1\}$), while \eqref{eq:l2bound} is a norm bound. Second, while copula bounds are -- by construction -- independent of marginal distributions, the norm bound established in this paper depends critically on the marginals densities considered.

The results of this paper are originally motivated by questions arising in the optimisation of options portfolios \citep{guasoni2017options}, which are equivalent, by convex duality, to the minimisation of the second moment of the stochastic discount factor, subject to constraints on its marginals. This paper investigates the more challenging setting of an arbitrary moment (as opposed to the second moment), which leads to two related issues. First, the problem cannot be tackled with Hilbert space techniques, but requires arguments from Banach space theory. Second, while the second moment leads to linear first-order conditions and a minimal density that is additive across different variables, the general case considered here entails as first-order conditions nonlinear integral equations and a resulting more complex nonlinear decomposition of the minimal density.

The rest of the paper is organised as follows. Section \ref{sec3} sets up the general $\mathcal L^p$-problem and offers a heuristic derivation of the equations governing the lower estimate. The rigorous proof of the sharp lower bound follows in Theorem \ref{th: main4}. This general statement is then applied to the case $p=2$ in Section \ref{sec: l2}, and results in Theorem \ref{th: main2}. Section \ref{sec: reg} reinterprets Theorem \ref{th: main4} as a regularity result for nonlinear integral equations. Section \ref{sec: num} provides a numerical study of the minimal solutions to the integral equations in dimension two, discussing how the optimisers depart from the linear case $p=2$. Finally, Section \ref{sec: prob} explains the 
probabilistic implications for portfolio selection.

\section{$\mathcal L^p$-bounds}\label{sec3}
Let $1<p<\infty$ and $n\geq 2$. Denote by $\mathcal L^p([0,1]^n)$ the space of equivalence classes of Lebesgue-measurable functions $f$
on the unit hypercube $[0,1]^n$, for which
$\| f\|_p:=\left(\int \vert f(\xi)\vert^p d\xi\right)^{\frac{1}{p}}<\infty$. $\mathcal L^p$-spaces are strictly convex\footnote{This property is also called rotund, cf. \citep[Example 1.10.2 and Theorem 1.11.10]{megginson2012introduction}} in that for any $0<t<1$ and $f,\,g \in\mathcal L^p([0,1]^n)$ such that $\|f\|_p=\|g\|_p=1$ and $f\neq g$, it holds that $\|tf+(1-t)g\|_p<1$. As $\mathcal L^p$-spaces are also reflexive, the following Banach-space analog \citep[Corollary 5.1.19]{megginson2012introduction} of a familiar Hilbert space result \citep[Theorem 4.10]{rudin2006real} holds:
\begin{theorem}\label{th1: Lp}
If a normed space is rotund and reflexive, then each of its nonempty, convex, closed subsets has a unique minimal element of smallest norm.
\end{theorem}

Henceforth, for $1\leq i\leq n$ denote by $\xi_i$ the $i$-th coordinate of $\xi\in\mathbb R^n$ and by $\xi^{c}_i$ the $(n-1)$-vector, in which the latter is omitted, i.e., $\xi_i^c=(\xi_1,\dots,\xi_{i-1},\xi_{i+1},\dots,\xi_n)$. We shall integrate functions $f=f(\xi)$ either with respect to $\xi_i^c$, in which case we write $f_i(\xi_i):=\int f(\xi)d\xi_i^c$, or with respect to $\xi_i$, and then we write $\int f(\xi)d\xi_i$.

\subsection{Heuristic Derivation}
To minimise the $\mathcal L^p$-norm $\|h\|_p$ subject to the marginal constraints
\[
\int h(\xi)d\xi_i^c=g_i(\xi_i),\quad 1\leq i\leq n,
\]
consider the Lagrangian
\[
L=\frac{1}{p}\int \vert h(\xi)\vert^pd\xi-\frac{1}{n}\sum_{i=1}^n \int\Phi(\xi_i)\left(\int h(\xi)d\xi_i^c-g_i(\xi_i)\right)d\xi_i.
\]
Setting the directional derivatives equal to zero yields the first order conditions
\[
\sign(h(\xi))\vert h(\xi)\vert^{p-1}=\frac{1}{n}\sum_{i=1}^n \Phi_i(\xi_i),
\]
whence
\[
h(\xi)=\sign\left(\sum_{i=1}^n\Phi_i(\xi_i)\right)\left\vert \frac{1}{n}\sum_{i=1}^n \Phi_i(\xi_i) \right\vert^{\frac{1}{p-1}}.
\]
The marginal constraints imply
\begin{equation}\label{eq: constraints marginals p}
\int \sign\left(\frac{1}{n}\sum_{j=1}^n\Phi_j(\xi_j)\right)\left\vert 
\frac{1}{n}\sum_{j=1}^n\Phi_j(\xi_j)
\right\vert^{\frac{1}{p-1}} d\xi_i^c=g_i(\xi_i),\quad 1\leq i\leq n.
\end{equation}
To uniquely identify\footnote{For the uniqueness proof, see the proof of Theorem \ref{th: main4}.} the Lagrange multipliers $\Phi_i$ -- which are otherwise determined up to an additive constant -- it suffices to impose the conditions
\begin{equation}\label{eq: super constraint}
\int \Phi_i(\xi_i) d\xi_i=0,\quad 2\leq i\leq n.
\end{equation}
Note that these conditions are required only for $i\geq 2$.
\subsection{Main result}\label{sec: main}
The discussion begins with a characterisation of minimality in Banach spaces \citep[Theorem 4.21]{shapiro2006topics}:
\begin{lemma}\label{lemx}
Let $1<p<\infty$, $f\in\mathcal L^p$, and $Y$ be a closed subspace of $\mathcal L^p$. The following are equivalent:
\begin{enumerate}
\item $\|f\|_p\leq \|f+k\|_p$ for all $k\in Y$.
\item $\int\sign(f(\xi)) \vert f(\xi)\vert^{p-1}k(\xi) d\xi=0$ for all $k\in Y$.
\end{enumerate}
\end{lemma}
A function $f$ is \emph{orthogonal} to a subspace $Y$ if it satisfies any of the equivalent statements of Lemma \ref{lemx}.
\begin{lemma}\label{lem: lem2}
Let $1<p<\infty$, $f\in\mathcal L^q([0,1]^n)$, where $q=p/(p-1)$, and denote by
\begin{equation}\label{eq: super}
\mathcal N:=
\left\{\phi\in \mathcal L^p([0,1]^n)\Big| \int \phi(\xi)d\xi_i^c\equiv 0, \text{ for } 1\le i\le n\right\}.
\end{equation}
The following are equivalent:
\begin{enumerate}
\item \label{eq: part 1xx} $\int f(x)\phi(x)dx=0$ for all $\phi\in\mathcal N$.
\item \label{eq: part 2xx} $f(x)=\frac{1}{n}\sum_{i=1}^n \Psi_i(x_i)$, where $\Psi_i$ lie in $\mathcal L^q([0,1])$, $1\leq i\leq n$.
\end{enumerate}
\end{lemma}
\begin{proof}
The implication \eqref{eq: part 2xx} $\Rightarrow$ \eqref{eq: part 1xx} is straightforward. To show that \eqref{eq: part 1xx} $\Rightarrow$ \eqref{eq: part 2xx}, note that, by Jensen's inequality, for any $\phi\in \mathcal L^p([0,1]^n)$ and any $1\leq i\leq n$, $\int \phi(\xi)d\xi_i^c\in \mathcal L^p([0,1])$. Hence, for any $\phi\in \mathcal L^p([0,1]^n)$, 
\[
\widetilde\phi(\xi):=\phi(\xi)-\sum_{i=1}^n\int \phi(\xi)d\xi_i^c+(n-1)\int \phi(\eta)d\eta
\in\mathcal N,
\]
and Fubini's theorem yields
\[
\int \left(f(x)-\sum_{i=1}^n \int f(\xi)d\xi_i^c+(n-1)\int f(\xi)d\xi\right)\phi(x)dx=0.
\]
Thus by duality, 
\[
f(\xi)=\sum_{i=1}^n \int f(\xi)d\xi_i^c-(n-1)\int f(\xi)d\xi
\qquad\xi-\text{a.e.}.
\]
The functions $\Phi_i(\xi_i):=n\int f(\xi)d\xi_i^c-(n-1)\int f(\xi)d\xi$, $1\leq i\leq n$, are in $\mathcal L^q([0,1]^n)$, and they sum
to $f$, as claimed.
\end{proof}
The previous two Lemmas combine to the following:
\begin{corollary}\label{cor: 1}
Let $f\in\mathcal L^q([0,1]^n)$, and $\mathcal N\subset\mathcal L^p([0,1]^n)$ as defined in \eqref{eq: super}. The following are equivalent:
\begin{enumerate}
\item \label{eq: part 0x} $\sign(f)\vert f\vert^{1/(p-1)}$ is orthogonal to $\mathcal N$.
\item \label{eq: part 1x} $\int f(x)\phi(x)dx=0$ for all $\phi\in\mathcal N$.
\item \label{eq: part 2x} $f(x)=\frac{1}{n}\sum_{i=1}^n \Psi_i(x_i)$, where $\Psi_i$ lie in $\mathcal L^q([0,1]^n)$, $1\leq i\leq n$.
\end{enumerate}
\end{corollary}

\begin{theorem}\label{th: main4}
Let $p>1$. Any $g\in\mathcal L^p$ satisfies
\begin{equation}\label{eq: estimate a}
\int \vert g(\xi)\vert ^pd\xi\geq \int\vert\overline\Phi(\xi)\vert^{\frac{p}{p-1}}d\xi,
\end{equation}
where 
\[
\overline \Phi(\xi):=\frac{1}{n}\sum_{i=1}^n \Phi_i(\xi_i)
\]
and $\Phi_i$ are the unique solutions of the system of integral equations \eqref{eq: constraints marginals p}--\eqref{eq: super constraint}.

The estimate \eqref{eq: estimate a} is sharp, in that equality in \eqref{eq: estimate a}  holds, if and only if
\begin{equation}\label{gp}
g(\xi)=\sign\left(\overline \Phi(\xi)\right)\left\vert\overline \Phi(\xi)\right\vert^{ \frac{1}{p-1} }.
\end{equation}
\end{theorem}
\begin{proof}
By Jensen's inequality, $g_i:=\int g(\xi)d\xi_i^c\in\mathcal L^p([0,1])$ for $1\leq i\leq n$, hence
the set
\[
\mathcal M:=
\left\{h\in \mathcal L^p([0,1]^n) \Big| \int h(\xi) d\xi_i^c= g_i(\xi_i),
\quad 1\leq i\leq n\right\}
\]
is well-defined, and it is non-empty because $g\in\mathcal M$. The set is convex, by construction, and it is closed: Let $h_n\in\mathcal M$ and $\lim_{n\rightarrow \infty} h_n=h$ in $\mathcal L^p([0,1]^n)$. Then the sequence $(h_n)_{n\ge 1}$ is uniformly integrable, hence by Vitali's convergence theorem, $\xi_i$- almost everywhere,
\[
\int h(\xi)d\xi_i^c=\int \lim_{n\rightarrow\infty} h_n(\xi)d\xi_i^c=\lim_{n\rightarrow\infty }\int h_n(\xi)d\xi^c=g_i(\xi_i).
\]
Denote by $h_*$ the unique element in $\mathcal M$ of smallest norm\footnote{See Theorem \ref{th1: Lp}, and the paragraph preceding it concerning its applicability for $\mathcal L^p$ spaces.}. We claim that $h_*=g$, where $g$ is defined in \eqref{gp}. To this end, recall the function space defined in \eqref{eq: super}. By the minimality of $h_*$, it follows that for any $\varepsilon>0$ and any $\phi\in \mathcal N$
\begin{equation}\label{eq: min}
\|h_*\pm \varepsilon \phi\|_p^p-\|h_*\|_p^p\geq 0,
\end{equation}
and therefore, by Lemma \ref{lemx}\footnote{Note that $\vert h_*\vert^{p-1}\in L^q$, where $q=\frac{p}{p-1}$, hence the below pairing is finite, by H\"older's inequality.}
\[
\int \sign(h_*(\xi))\vert h_* (\xi)\vert^{p-1} \phi(\xi)d\xi=0,\quad \phi\in\mathcal N.
\]
 The implication \eqref{eq: part 0x} $\Rightarrow$ \eqref{eq: part 2x} in Corollary \ref{cor: 1} yields
\[
\sign(h_*(\xi))\vert h_*(\xi)\vert^{p-1}=\overline \Phi(\xi), \quad \text{where}\quad \overline\Phi(\xi):=\frac{1}{n}\sum_{i=1}^n \Phi_i(\xi_i),
\]
with measurable functions $\Phi_i(\xi_i)$, $1\le i\le n$, depending on one variable $\xi_i$ only. Because $\sign(h_*)=\sign(\overline\Phi(\xi))$, it follows that
\[
h_*(\xi_1,\dots,\xi_n)=\sign( \overline\Phi(\xi))\left\vert\overline\Phi(\xi)\right\vert^{\frac{1}{p-1}}
\]
and $\overline \Phi$ solves the nonlinear integral equations \eqref{eq: constraints marginals p} for $1\leq i\leq n$. As these equations involve the sum $\overline \Phi$ only, we can satisfy the extra constraints \eqref{eq: super constraint}, by replacing $\Phi_i$ by $\Phi_i-\int \Phi_i(\xi_i)d\xi_i$ ($2\leq i\leq n$), if necessary.

It remains to show the uniqueness. Assume that, besides $\overline \Phi$, also $\overline\Psi(\xi):=\frac{1}{n}\sum_{i=1}^n\Psi_i(\xi_i)$ solves \eqref{eq: constraints marginals p}--\eqref{eq: super constraint}. By Corollary \ref{cor: 1} \eqref{eq: part 2x} $\Rightarrow$ \eqref{eq: part 0x}, the function $h:=\sign(\overline\Psi)\vert \overline\Psi\vert^{\frac{1}{p-1}}$
is orthogonal to $\mathcal N$ defined in \eqref{eq: super}. Furthermore, by \eqref{eq: constraints marginals p}, $h-h_*\in\mathcal N$, hence by definition of orthogonality, 
$\|h\|_p\leq \|h_*\|_p$. In view of \eqref{eq: min}, $h=h_*$, whence also $\overline \Psi=\overline \Phi$. As $\overline\Phi(\xi)=\frac{1}{n}\sum_{i=1}^n \Phi_i(\xi_i)=\frac{1}{n}\sum_{i=1}^n\Psi_i(\xi_i)=:\overline\Psi(\xi)$ almost everywhere, the extra constraints \eqref{eq: super constraint} yield, upon integration of $n\overline \Phi=n\overline\Psi$ with respect to $d\xi_1^c$, that 
$\Phi_1(\xi_1)=\Psi_1(\xi_1)$ almost everywhere (the rest of the integrals vanish). Applying the constraint for $i=2$, it follows that
\[
\int \Phi_1(\xi_1)d\xi_1+\Phi_2(\xi_2)+0=\int \Psi_1(\xi_1)d\xi_1+\Psi_2(\xi_2)+0=\int \Phi_1(\xi_1)d\xi_1+\Psi_2(\xi_2),
\]
whence $\Phi_2(\xi_2)=\Psi_2(\xi_2)$ $\xi_2$-almost everywhere. Continuing similarly for $3\le i\le n$, it follows that $\Phi_i=\Psi_i$ $\xi_i$-almost everywhere for $3\leq i\leq n$.
\end{proof}

\begin{remark}\label{rem: constant marginals}
Suppose all marginals are densities. If $n-1$ marginal data are uniform densities, say $g_i(\xi_i)=1$ for $i\in \mathcal T\subset\{1,2,\dots,n\}$ with $\vert T\vert\geq n-1$, then the corresponding
solutions $\Phi_i$ ($i\in T$) are constants. This is easiest to see for $T=\{2,\dots,n\}$. Then, it can be verified that the solution of \eqref{eq: constraints marginals p}--\eqref{eq: super constraint} is of the form
$\Phi_i=0$ for $2\leq i\leq n$, and $\Phi_1:=\sign(g_1)\vert g_1\vert^{p-1}$. If the index set $T$ contains $1$, then $\Phi_1$ will be a constant, but not necessarily equals to $0$ (see last row of Figure \ref{figall}).

\end{remark}

\subsection{The $\mathcal L^2$-case}\label{sec: l2}
For $p=2$, the system of integral equations \eqref{eq: constraints marginals p}--\eqref{eq: super constraint} becomes linear:
\begin{align}\label{eqy1}
\left(\int\sum_{j=1}^n\Phi_j(\xi_j)\right)d\xi_i^c&=n g_i(\xi_i),\quad 1\leq i\leq n,\\\label{eqy2}
\int \Phi_i(\xi_i)d\xi_i&=0,\quad 2\leq i\leq n.
\end{align}
For $i=1$, Equation \eqref{eqy1} yields $\Phi_1(\xi_1)=n g_1(\xi_1)$, and in conjunction with Equation \eqref{eqy2}, it follows that
\[
\Phi_i(\xi_i)=n \left( g_i(\xi_i)-\int g_1(\xi_1)d\xi_1\right )=n \left( g_i(\xi_i)-\int g(\xi)d\xi\right ),\quad 2\leq i\leq n.
\]
In view of Theorem \ref{th: main4}, it follows that:
\begin{theorem}\label{th: main2}
Any $g\in\mathcal L^2([0,1]^n)$ satisfies
\begin{equation}\label{eq: estimate1}
\int g^2(\xi)d\xi\geq 
\sum_{i=1}^n 
\int g_i(\xi_i)^2 d\xi_i
-(n-1)\left(\int g(x)dx\right)^2.
\end{equation}
The bound is sharp: equality holds if and only if $g$ equals
\[
\overline \Phi(\xi):=\left(\sum_{i=1}^n g_i(\xi_i)\right)-(n-1)\int g(x)dx .
\]
\end{theorem}
\begin{remark}
Note that the above includes the case where $g$ has vanishing integral. In this case, the estimate \eqref{eq: estimate1}
lacks the last term $(n-1)\int g(x)dx$.
\end{remark}
\subsection{Regularity of Integral Equations}\label{sec: reg}
One can view Theorem \ref{th: main4} as a regularity result for nonlinear integral equations with constraints -- the dual problem:
\begin{corollary}
Let $n\geq 2$ and $p>1$. For any $g\in\mathcal L^p([0,1]^n)$, there exists a unique solution 
\[
\Phi(\xi)=(\Phi_1(\xi_1),\Phi_2(\xi_2),\dots,\Phi_n(\xi_n))
\]
of the integral equations \eqref{eq: constraints marginals p}--\eqref{eq: super constraint} with data $g$, satisfying the bound
\[
\|\overline \Phi\|_{q}\leq \|g\|_p,
\]
where $q$ denotes conjugate exponent (satisfying $1/p+1/q=1$). 
\end{corollary}

\section{Numerical Examples}\label{sec: num}

For illustration, we study dimension $n=2$ and consider data $g_1(\xi_1)$ and $g_2(\xi_2)$ of unit integral.

To solve the nonlinear integral equations \eqref{eq: constraints marginals p}--\eqref{eq: super constraint}, we discretise the involved integrals using a mesh-size of $1/30$, 
and solve the resulting nonlinear equations with the optim solver on R (minimising residuals), using the exact solution in the $p=2$ case as starting value; and then increasing (resp. decreasing) from the Hilbertian case by $\Delta p=10\%$, using repeatedly the previous numerical solution as seed for the solver. Each example involves normal marginal densities, re-normalised so to have unit mass on $[0,1]$.\footnote{It may appear more natural to pick densities that integrate to one by default, e.g. instances of the beta distribution with parameters $\alpha,\beta$. However, this seemingly more natural choice leads to similar effects as explained below. For intuitive and for numerical reasons, we found normal data more convenient.}

Figure \ref{figall} below depicts solutions for three different sets of data $g_1$, $g_2$, and for a range of $p>1$. The explicit solution for $p=2$ (Theorem \ref{th: main2})
 \[
\Phi_1(\xi_1)=2g_1(\xi_1),\quad \Phi_2(\xi_2)=2 (g_2(\xi_2)-1).
\]
is plotted by dotted lines. The Figure generally shows that as $p\downarrow 1$, the solutions $\Phi_1,\Phi_2$ get flatter (dashed lines), while as $p$ increases from $2$, the solutions have more and more pronounced peaks (solid lines). In the first row of Figure \ref{figall} the data is identical, however since $\int\Phi_2(\xi_2)d\xi_2=0$, $\Phi_1$ and $\Phi_2$ differ by a real constant. Moreover, the second row of the Figure uses inhomogeneous data $g_1, g_2$, centering around different means. Finally, the last row of Figure \ref{figall} confirms the theoretical finding of Remark \ref{rem: constant marginals}: the solution $\Phi_1$ on the left, that corresponds to uniform marginal data $g_1$, is constant also for $p\neq 2$.

\begin{figure}[h]
\begin{center}
\includegraphics[width=0.9\textwidth]{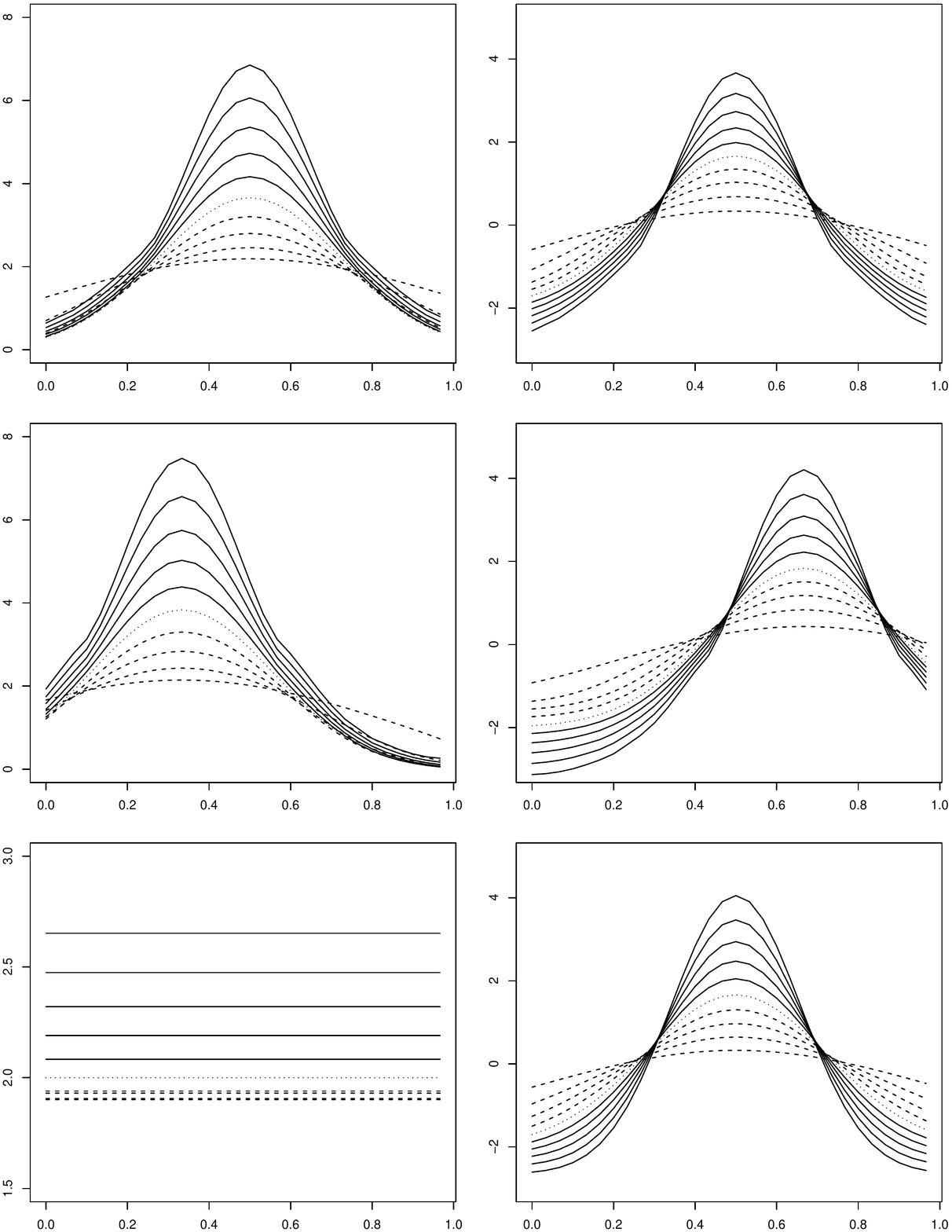}\label{fig6}
\caption{In each row, we depict the solutions of the nonlinear integral equations \eqref{eq: constraints marginals p}--\eqref{eq: super constraint}; on the left panels $\Phi_1$ and on the right one $\Phi_2$, respectively. The dotted lines depict the explicit solution in the $p=2$ case; the other lines depart from the Hilbertian case by $\Delta p=20\%$ steps; the dashed lines depict the solutions for $p=1.8,1.6,1.4$ and $1.2$. Similarly, the solid lines depict the numerical solutions for $p=2.2, 2.4,\dots,3$. In the first row, both marginals are Gaussian densities with equal parameter $\mu=1/2$ and $\sigma^2=10\%$, truncated and re-normalised so  to have unit mass on $[0,1]$.  In the second row, the means are varied to $\mu=1/3$ on the left, and $\mu=2/3$ on the right. In the third row, $g_1=1$ (uniform density). }
\label{figall}
\end{center}
\end{figure}

\section{Probabilistic Implications}\label{sec: prob}
The main result in this paper is originally motivated by the problem of selecting the portfolio of options with maximal Sharpe ratio, which is equivalent, by convex duality, to the minimisation of the second moment of the stochastic discount factor, subject to constraints on its marginals (cf. \citep [Section 2.2]{guasoni2017options}). In this financial application, the baseline measure is described by a probability density $p(\xi)$, which identifies the views of an investor on the joint distribution of the prices of $n$ risky assets $\mathcal S^1,\dots, \mathcal S^n$ at some future date $T$ (modelled by random variables $S^1,\dots, S^n$).  Each of these is an underlying asset of European Call and Put options with a continuum of strikes, maturing at $T$. The density $p$ thus replaces the Lebesgue measure in the present paper. The corresponding $\mathcal L^p$-problem ($p>1$) generalises equations \eqref{eq: constraints marginals p}--\eqref{eq: super constraint} to
\begin{align*}
\int \sign\left(\frac{1}{n}\sum_{j=1}^n\Phi_j(\xi_j)\right)\left\vert \frac{1}{n}\sum_{j=1}^n\Phi_j(\xi_j)\right\vert^{\frac{1}{p-1}} p(\xi)d\xi_i^c=g_i(\xi_i),\quad 1&\leq i\leq n,\\
\int \Phi_i(\xi_i) p_i(\xi_i)d\xi_i=0,\quad 2&\leq i\leq n.
\end{align*}
The new weight $p$ in these equations, and its $n$ marginal densities $p_i$ constitute only a small modification that can be treated similarly as in Section \ref{sec: main}. The ``inhomogeneous data'' $g_i$ ($1\leq i\leq n$) represents the risk-neutral marginal distributions of asset $S^i$ ($1\leq i\leq n$) which is, due to \citep{breeden1978prices,nachman1988spanning,green1987spanning} proportional to the second derivative of European Call options prices,
with respect to their strike price. 

An important implication of the result is that the solution is necessarily of the functional form $g(\xi)=\sign(\overline \Phi(\xi)\left| \overline\Phi(\xi)\right|^{1/(p-1)}$, where $\overline\Phi(\xi)=\frac{1}{n}\sum_{j=1}^n\Phi_j(\xi_j)$. First, this result casts in a cautionary light the practice of pricing complex derivatives by fitting parametric copulas to risk-neutral
marginals. In general, the price obtained from some copula family is neither replicable nor linked to an optimisation objective.
Second, the result implies that the maximal random payoff $\overline\Phi$ (the primal object) is not an arbitrary function of the the asset prices $S^1,\dots, S^n$, but it additively separates in each asset, that is, it is of the form
\[
\overline \Phi(S_1,\dots,S_n)=\frac{1}{n}\sum_{j=1}^n\Phi_j(S_j),
\]
and thus can be interpreted as an option portfolio, featuring European options on each of the $n$ individual underlyings. Indeed, according to the Carr-Madan formula \citep{carr2001towards}, each summand $\Phi_j$, if regular enough, can be written as
\begin{align*}
\Phi_j(K)&=\Phi_j(K_0)+\Phi_j'(K_0)(K-K_0)+\int_0^{K_0}\Phi_j''(\kappa)(\kappa-K)^+d\kappa\\&\qquad+\int_{K_0}^\infty \Phi_j'(\kappa)(K-\kappa)^+d\kappa,\quad K\geq 0
\end{align*}
and thus decomposes into the payoff of $\Phi_j(K_0)$ units of a zero-coupon bond maturing at $T$ (first summand), $\Phi_j'(K_0)$ forward contracts with delivery price $K_0$ (second summand) and a continuum of European puts (with strike $\kappa\leq K_0$), and European calls (with strikes $\kappa\geq K_0$).\footnote{The strike $K_0>0$ that separates puts from calls is arbitrary, and typically equal to the spot or forward price.}

\section*{Acknowledgements}
Paolo Guasoni was partially supported by SFI (16/IA/4443,16/SPP/3347). Eberhard Mayerhofer thanks Natalia Kopteva for suggestions concerning Section \ref{sec: reg}. Mingchuan Zhao acknowledges funding through SFI/12/RC/2289 P2.
\bibliographystyle{plainnat}

\end{document}